\documentclass[doublespacing,10pt,reqno]{amsart}
\usepackage[centertags]{amsmath}

\theoremstyle{plain} 

\newtheorem{lem}{Lemma}

\theoremstyle{definition}

\theoremstyle{remark}

\DeclareMathOperator{\ad}{ad} \DeclareMathOperator{\Ad}{Ad}

 \DeclareMathOperator{\plg}{Pr}

\DeclareMathSymbol{\R}{\mathalpha}{AMSb}{"52}
\DeclareMathSymbol{\C}{\mathalpha}{AMSb}{"43}

\newcommand{\beq}{\begin{equation}}
\newcommand{\eeq}{\end{equation}}

\newcommand{\set}[1]{\left\{#1\right\}}
\newcommand{\seq}[1]{\left<#1\right>}

\newcommand{\pd}{\partial}

\newcommand{\ra}{\rightarrow}

\newcommand{\bd}{\begin{description}}
\newcommand{\ed}{\end{description}}
\newcommand{\beqr}{\begin{eqnarray}}
\newcommand{\eeqr}{\end{eqnarray}}

\begin{document}

\title[Symmetry properties of an acoustic**(s) model] 
{Symmetry properties of a nonlinear acoustics
model}
\author[J C Ndogmo]{ J C  Ndogmo}

\address{
P O Box 2446\\
Bellville 7535\\
South Africa.}
\email{jndogmo@uwc.ac.za}

\begin{abstract}
We give a classification into conjugacy classes of subalgebras of
the symmetry algebra generated by the Zabolotskaya-Khokhlov
equation, and obtain all similarity reductions of this equation into
$(1+1)$-dimensional equations. We thus show that Lie classical
reduction approach may also give rise to more general reduced
equations as those expected from the direct method of Clarkson and
Kruskal. By transforming the determining system for the similarity
variables into the equivalent adjoint system of total differential
equations, similarity reductions to {\sc ode}s which are independent
of the three arbitrary functions defining the symmetries are also
obtained. These results are again compared with those obtained by
the direct method of Clarkson and Kruskal, by finding in particular
equivalence transformations mapping some of the reduced equations to
each other. Various families of new exact solutions are also
derived.
\end{abstract}


%
%
\maketitle

\section{Introduction}

The Zabolotskaya-Khokhlov (ZK) equation is a nonlinear model of
sound wave propagation derived from the incompressible Navier-Stokes
equation \cite{pavp29, kostin2, kostin3}. The $(2+1)$- dimensional
version of this equation  has the form

\beq \label{eq:zk}\Delta (t,x,y,u) \equiv u_{xt}- (u u_x)_x - u_{y
y}= 0, \eeq

and it has been studied from the Lie group approach in a number of
papers \cite{chowdhu, schw, hereman, zhu}. Chowdhury and Naser
\cite{chowdhu} attempted the determination of the symmetry algebra
of this equation and calculated some of their conservation laws.
However, Schwarz \cite{schw} and Hereman \cite{hereman} were
certainly the first to provide independently the correct generators
of the Lie symmetry algebra for this popular acoustics model.
Although the calculation of symmetry algebras for almost all systems
of differential equations has long been reduced to a mere function
on a number of modern computing systems, many symmetry properties of
this equation are still to be uncovered. \par

In \cite{zhu}, some similarity reductions of Equation ~\eqref{eq:zk}
to $(1+1)$-dimensional models were obtained, based on the direct
method of Clarkson and Kruskal  \cite{clarkson}. More specifically,
it was shown that if a similarity solution of Equation
~\eqref{eq:zk} of the form
\beq \label{eq:pdevars} u= U(t,x,y, W(\xi, \eta)), \qquad
\xi=\xi(t,x,y), \quad\ \eta =\eta(t,x,y) \eeq
can be found by solving a reduced $(1+1)$-dimensional equation,
then, when $\xi_x \neq 0,$ such an equation must be one of three
non-equivalent $(1+1)$-dimensional equations found in that paper.
However, these reduced equations are either determined only
implicitly in terms of solutions of certain partial differential
equations ({\sc pde}s), or they depend on up to four arbitrary
functions, and the same can be said about their solutions. In the
same paper, using again the direct method of Clarkson and Kruskal
and a restrictive anzatz, the most general {\sc ode} that every
similarity solution of ~\eqref{eq:zk} obtained by solving an {\sc
ode} must satisfy is shown to be of the form

\beq \label{eq:zhu} w'^2 + w w'' + (A z + B) w' + 2 A =\frac{1}{3}
(A z + B)^2 . \eeq
Some of the unanswered questions raised  in \cite{zhu} were how to
find all the non equivalent similarity reductions of the ZK equation
to an {\sc ode} by the classical Lie symmetry approach, and whether
there is any connection between these two types of reduction
techniques. The first of these two questions stems from the fact
that the symmetries of Equation ~\eqref{eq:zk} depend on three
arbitrary functions of time, and so its similarity reductions are
usually achieved by restricting these functions to some specific
types of elementary functions, such as exponential or simple
polynomial functions of time \cite{chowdhu, schw}.  It also stems
from the fact that no classification of low-dimensionial subalgebras
of the ZK symmetry algebra into conjugacy classes is available.\par

   In this paper, we obtain all canonical forms of non-equivalent
one- and two-dimensional subalgebras of the symmetry algebra $L$ of
~\eqref{eq:zk}, under the adjoint representation of the symmetry
group. We then apply them to obtain all similarity reductions of the
ZK equation to $(1+1)$-dimensional equations, using the classical
Lie approach. The same reductions are also obtained by direct case
analysis. We thus show that in addition to its simpler algorithm and
other properties, Lie classical reduction method gives rise not only
to simpler equations, but  it may also yield more unified and
general reduced equations than those expected from the direct method
(of Clarkson Kruskal). Next, by transforming the determining system
for the similarity variables into the equivalent adjoint system of
total differential equations, similarity reductions to large classes
of {\sc ode}s which are independent of the three arbitrary functions
defining the symmetries are also obtained. The latter system of
total differential equations allows for an easier determination of
the invariant functions defining the similarity coordinates.  Large
classes of similarity solutions depending on much less arbitrary
functions than those obtained in \cite{zhu} are also derived in this
way.
\par
   Finally, we find equivalence
transformations mapping some of the equations that we've obtained by
the Lie classical method to some sub-equations of the reduced
equation ~\eqref{eq:zhu} obtained by the direct method of Clarkson
and Kruskal. Our discussions also show that in principle any reduced
equation achievable with the direct method can also be achieved by
Lie classical method although the converse is totally out of
question, as far as the properties of the reduced equations are
concerned.\par

This paper is organized as follows. In the next section we discuss
the symmetry algebra of the ZK equation and determine its algebraic
structure as well as its connection with Kac-Moody-Virasoro ({\sc
kmv}) algebras.  Section ~\ref{s:classification} is devoted to the
classification of low-dimensional subalgebras of $L$ and Section
~\ref{s:reduction} to the similarity reductions of the ZK equation.
We investigate the connection between the two types of reductions
invoked above in Section ~\ref{s:comparison} . Some concluding
remarks are given in the last section.

\section{Symmetry group of the ZK equation}
\label{s:symg}
\subsection{Structure of the symmetry algebra}
\label{sb:structure} The Lie algebra of the ZK equation is well
known \cite{schw, hereman}. This is the Lie algebra defined by the
infinitesimal generators of the point symmetry group $G$ of the
equation, that is, the Lie group of point transformations that map
every solution of the equation to another solution of the same
equation \cite{olv1, stephani}. These infinitesimal generators are
vector fields of the form
\beq \label{eq:vectf} \mathbf{v}= \xi_1 (t,x,y,u) \partial_{t}  +
\xi_2 (t,x,y,u)
\partial_{x}  + \xi_3 (t,x,y,u)
\partial_{y} + \phi (t,x,y,u) \partial_{u}
\eeq
acting on the space of independent variables coordinatized by $(t,
x, y)$  and the space of the dependent variable coordinatized by
$u,$ and such that the second prolongation $\plg^{(2)} \mathbf{v}$
of $\mathbf{v}$ satisfies

\beq \label{eq:plg}\plg^{(2)}\mathbf{v} \; \Delta (t,x,y,u)
\Big\vert_{\, \Delta (t,x,y,u) =0}=0. \eeq

 Equation ~\eqref{eq:plg} completely determines the unknown functions
 $\xi_1, \xi_2, \xi_3, \phi$ defining $\mathbf{v}$ \cite{olv1, olv2, stephani}.
 For the ZK equation, the linearly independent vector fields, i.e. the
 generators of the Lie symmetry algebra are given by the operators
\begin{subequations}\label{eq:gtors}
\begin{align}
\mathbf{v}_0 \; &= \; 2 x \partial_x + y \partial_y + 2 u \partial_u \\
\mathbf{x}_g \; &=\; g \partial_x - g' \partial_u \\
\mathbf{y}_h \; &=\;  \frac{1}{2} y h'\partial_x + h\partial_y
-\frac{1}{2} y
h'' \partial_u \\
\mathbf{z}_f \;&=\; f \partial_t + \frac{1}{6} \left(2 x f' + y^2
f''\right) \partial_x + \frac{2 y}{3}  f' \partial_y + \frac{1}{6}
\left(- 4 u f' - 2 x f'' - y^2 f'''\right) \partial_u
\end{align}
\end{subequations}
 where $f, g, h$ are arbitrary $C^{\infty}$ functions of the time
 variable $t$ defined on some open subset of $\R,$ assumed to be the
based field of $L,$ and where a prime represents a derivative with
respect to $t.$ The ZK symmetry
 algebra is therefore infinite-dimensional, and its commutation
 relations are easily found to be as follows
\begin{subequations}\label{eq:com}
\begin{alignat}{2}
[\mathbf{v}_0, \mathbf{x}_g]\;  &= \; -2 \mathbf{x}_g,\medspace &
[\mathbf{x}_g, \mathbf{y}_h]\; &= \;
0  \\
[\mathbf{v}_0, \mathbf{y}_h]\;  &= \; - \mathbf{y}_h,\medspace  &
[\mathbf{x}_g, \mathbf{z}_f]\;
&= \; \mathbf{x}_{( f' g/3 - f g')}   \\
[\mathbf{v}_0, \mathbf{z}_f]\; &= \;0, \medspace & [\mathbf{y}_h, \mathbf{z}_f]\; &= \; \mathbf{y}_{ (\frac{2}{3} f'h - f h')} \\
[\mathbf{x}_{g_1}, \mathbf{x}_{g_2}] \; &= \; 0  \medspace &  [\mathbf{y}_{h_1}, \mathbf{y}_{h_2}] \; &= \; \mathbf{x}_{ (h_1 h_2'- h_1' h_2 )/2}  \\
[\mathbf{z}_{f_1}, \mathbf{z}_{f_2}]\; &= \; \mathbf{z}_{ (f_1 f_2'
-f_1' f_2 )}. \medspace & &
\end{alignat}
\end{subequations}
These commutation relations show that the ZK symmetry algebra $L$
has the structure of a semi-direct sum Lie algebra, $L= \mathcal{R}
+ \mathcal{S},$ where $\mathcal{R} = \seq{\mathbf{v}_0,
\mathbf{x}_g, \mathbf{y}_h}$ is the radical and $\mathcal{S}=
\seq{\mathbf{z}_f}$ is the semisimple part, also called Levi factor
of $L.$ It is indeed a well-known fact \cite{pavkp} that the
infinite dimensional Lie algebra generated by $\mathbf{z}_f$  is a
simple Lie algebra, i.e. it has no nontrivial ideal. This is easily
established by first nothing that by a result of Cartan
\cite{cartan24}, the Lie algebra $J(I)= \set{ f(t) \partial_t \colon
f \in C^{\infty}(I)}$ of vector fields on the open interval $I$ of
$\R$ is a simple algebra, and that the map
$$
\sigma \colon J(I) \ra \mathcal{S} \colon f(t) \partial_t\mapsto
\mathbf{z}_{f(t)}
$$
is a Lie algebra isomorphism. We also note that the radical
$\mathcal{R},$ which is solvable by definition, is actually
nonnilpotent. Its nilradical, i.e. its maximal nilpotent ideal is
generated by $\seq{\mathbf{x}_g, \mathbf{y}_h}.$ \par The
commutation relations ~\eqref{eq:com} also show that the
infinite-dimensional subalgebra $L_0 = \seq{\mathbf{x}_g,
\mathbf{y}_h, \mathbf{z}_f}$ of $L$ can be identified with a
subalgebra of a {\sc kmv} type algebra. Indeed, by restricting the
arbitrary functions $f, g,$ and $h$ to be Laurent polynomials, the
resulting commutation relations take the form
\begin{subequations}\label{eq:kmv2}
\begin{alignat}{2}
[\mathbf{x}_{t^m}, \mathbf{x}_{t^n}] \; &= \; 0,  \qquad &
[\mathbf{x}_{t^m}, \mathbf{y}_{t^n}] \; &= \; 0  \\
[\mathbf{z}_{t^m}, \mathbf{z}_{t^n}]\; &= \; (n-m)\,
\mathbf{z}_{t^{m+n-1}},               \qquad &
[\mathbf{x}_{t^m}, \mathbf{z}_{t^n}]\; &= \; \frac{m-3n}{3}\, \mathbf{x}_{t^{m+n-1}}  \\
[\mathbf{y}_{t^m}, \mathbf{z}_{t^n}]\; &= \; \frac{2m-3n}{3}\,
\mathbf{y}_{t^m+n-1},
                   \qquad & [\mathbf{y}_{t^m},\mathbf{y}_{t^n}] \; &= \;
\frac{n-m}{2}\, \mathbf{x}_{t^{m+n-1}},
\end{alignat}
\end{subequations}
and this shows that $\seq{\mathbf{x}_g, \mathbf{y}_h}$ generates the
corresponding Kac-Moody subalgebra of $L_0,$ while
$\seq{\mathbf{z}_f}$ generates the  Virasoro subalgebra
\cite{pavpre1}. These algebraic properties of the ZK equation, which
can be linearized by a generalized hodograph transformation
\cite{gibbons}, are in agreement with a widespread belief about
integrable $(2+1)$-dimensional equations. Indeed, the symmetry
algebra of most integrable $(2+1)$-dimensional equations are
infinite-dimensional and can be identified with a subalgebra of a
{\sc kmv} type algebra \cite{pavkp, pavpre1, gungor}. There are
nonetheless some exceptions provided for instance by the breaking
soliton equation and the Zakharov-Strachan equation \cite{sentil}
which are integrable but do not admit a {\sc kmv} type symmetry
algebra. On the other hand, the symmetry algebra of every known
non-integrable equations has no {\sc kmv} structure. This suggests
that {\sc kmv} structures are in some way which is still to be
clarified  associated with integrability. By an integrable equation
here, we refer to equations  allowing a Lax pair, an infinity of
conservation laws, soliton and multisoliton solutions, a family of
analytic periodic and quasi-periodic solutions, as well as a number
of similar properties.

\subsection{Group transformations}
\label{sb:gptransfo} One of the most important properties of the
symmetry group of a differential equation is to transform one
solution of the equation into another solution of the same
differential equation. These group transformations are generated by
some one-parameter group of transformations, each of which is the
local flow $\psi (\epsilon ,\mathbf{w}),$ where
$\mathbf{w}=(t,x,y,u)$ and $\epsilon \in \R,$ of a basis element
$\mathbf{v}$ of the Lie symmetry algebra, acting on the space of
independent and dependent variables. More specifically, the local
flow $\psi (\epsilon, \mathbf{w})$ of $\mathbf{v}$ is given for
every $\epsilon$ for which it is defined by

\beq \label{eq:flow}  \partial_{\epsilon}\, \psi (\epsilon,
\mathbf{w}) \equiv \dot{\psi}(\epsilon, \mathbf{w}) = \mathbf{v}
\big \vert_{\,\psi (\epsilon, \mathbf{w})}, \quad \text{ and } \quad
\psi (0, \mathbf{w})= \mathbf{w}, \eeq

where $\partial_{\epsilon}\, =  d/d \epsilon.$ The full group action
can be obtained by combining all the one-parameter group actions
determined by each infinitesimal generator. We will use the notation
$$ \psi (\epsilon, \mathbf{w})= \tilde{\mathbf{w}} =
(\tilde{t}, \tilde{x}, \tilde{y}, \tilde{u}),$$ and we let the
generic generator of the ZK symmetry algebra have the form
\beq \label{eq:genvect}  \mathbf{v}= k_0  \mathbf{v}_0 +
\mathbf{x}_g + \mathbf{y}_h+ \mathbf{z}_f,\eeq
where $k_0$ is a scalar. We also denote by $F'$ the derivative of a
function $F=F(t)$ of the time variable $t.$ To find the flow of
$\mathbf{v}$ for the ZK equation, we
have to distinguish a number of cases.\\[1mm]

\underline{\bf Case 1} $\colon$ $g=h=f=0,$ and $k_0 \neq 0$\\
In this case $\mathbf{v}= \mathbf{v}_0,$ and by assuming without
loss of generality that $k_0=1,$ we have
\begin{align*}
\partial_{\epsilon}\, \tilde{t} &= 0  &\qquad  \partial_{\epsilon}\,
\tilde{y} &= \tilde{y} \\
\partial_{\epsilon}\, \tilde{x}&= 2 \tilde{x}  &\qquad  \partial_{\epsilon}\,
\tilde{u} &= \tilde{u}.
\end{align*}
Consequently, the corresponding group action is given by
$$
\psi (\epsilon, \mathbf{w})= (t, x e^{2 \epsilon}, y e^{ \epsilon},
u e^{2 \epsilon}) = \tilde{\mathbf{w}}.
$$
It then follows that
$$
\tilde{u}(t,x,y) = e^{2 \epsilon} u(t, x e^{- 2 \epsilon}, y e^{-
\epsilon} )
$$
 is a solution of the ZK equation whenever $u(t,x,y)$ is a
 solution.\\[1mm]

\underline{\bf Case 2}$\colon \quad f=h=0, \quad g \neq 0, \;k_0$ is
arbitrary \\
We have in this case $\mathbf{v}= k_0 \mathbf{v}_0 + \mathbf{x}_g.$\\[1mm]

$\qquad$ \underline{Case 2a} $\colon \quad k_0=0$\\
We find that\\
$$ \tilde{\mathbf{w}} = (t,\, x+ \epsilon g, \, y, \, u- \epsilon g'). $$
Consequently,
$$
\tilde{u} (t,x,y)= u(t, x- \epsilon g, y) - \epsilon g'
$$
is a solution whenever $u(t,x,y)$ is.\\[1mm]

$\qquad$\underline{Case 2b} $\colon \quad k_0 \neq 0$\\
In this case, $\tilde{\mathbf{w}}$ is given by \\[-3mm]
\begin{align*}
\tilde{t} &= t  \\
\tilde{x} &= \frac{1}{2k_0} (e^{2 k_0  \epsilon} (2 k_0 x +g) - g)  \\
\tilde{y} &=   y e^{k_0 \epsilon}\\
\tilde{u} &= \frac{1}{2k_0}(e^{ 2 k_0 \epsilon} ( 2 k_0 u - g')+
g').
\end{align*}
 Therefore, if $u(t,x,y)$ is a solution, then
$$
\tilde{u} (t,x,y) = \frac{1}{2k_0} \left[ e^{2 k_0 \epsilon} (2 k_0
u^* -g') + g' \right]
$$
is also a solution, where
$$
u^*= u\left[t, \frac{1}{2k_0} \left( (2 k_0 x + g)e^{-2 k_0
\epsilon} - g\right), y e^{- k_0 \epsilon}\right].
$$
\underline{\bf Case 3} $\colon \quad f=0, h \neq 0$ and $g$ and
$k_0$
are arbitrary\\
 We have two subcases in this case.\\[1mm]

$\qquad$ \underline{Case 3a} $\colon \quad k_0=0$\\
We have
$$
\tilde{\mathbf{w}}= (t, x+ g \epsilon+ \epsilon(y + h \epsilon)h', y
+ 2 h \epsilon, u - \epsilon (g' +  (y+ h \epsilon) h'') ).
$$

Consequently, whenever $u(t, x, y)$ is a solution, so is
$$
\tilde{u} (t,x,y) = u^*- \epsilon (g' + y (- h \epsilon)h'')
$$
where
$$
u^*= u (t, x-g \epsilon + \epsilon(- y + h \epsilon)h', y- 2 h
\epsilon).
$$

$\qquad$ \underline{Case 3b} $\colon \quad k_0 \neq 0$\\

We find that  $\tilde{\mathbf{w}}$ is given by
\begin{align*}
\tilde{t} &= t  \\
\tilde{x} &= \frac{1}{2 k_0^2} \left[  - k_0 g  + 2 h h'' + e^{k_0
\epsilon}
            (- 4 h h' - 2 k_0 y h')
+ e^{2 k_0 \epsilon} (k_0 g + 2 k_0^2 x + 2 h h' + 2 k_0 y h') \right]  \\
\tilde{y} &=  \frac{1}{k_0} \left[  - 2 h + e^{k_0 \epsilon} (2 h + k_0 y)\right] \\
\tilde{u} &= \frac{-1}{ 2 k_0^2} \left[ - k_0 g' + 2 h h'' + e^{k_0
\epsilon} (- 4 h h'' - 2 k_0 y h'') + e^{2 k_0 \epsilon} (- 2 k_0^2
u + k_0 g' + 2 h h'' + 2 k_0 y h'') \right].
\end{align*}

This shows that whenever $u(t,x,y)$ is a solution, so is
$$
\tilde{u} ( \tilde{t}, \tilde{x}, \tilde{y})= \frac{1}{2
k_0^2}\left[ 2 k_0^2 e^{2 k_0 \epsilon} u (t,x,y) - k_0 (-1+ e^ {2
k_0 \epsilon}) g'+ 2 (e^{k_0 \epsilon} -1) ( (-1+ e^{k_0 \epsilon})h
- k_0 \tilde{y} ) h'' \right]
$$
where
\begin{align*}
t &= \tilde{t} \\
x &= \frac{1}{2 k_0^2} \left[ e^{- 2 k_0 \epsilon} (k_0 (g- e^{2 k_0
\epsilon})g + 2 k_0  \tilde{x}) + 2 h'(-1+ e^{k_0 \epsilon})((e^{k_0
\epsilon} -1)h - k_0 \tilde{y}) \right]  \\
y &=  -\frac{1}{k_0} \left[   e^{-k_0 \epsilon} (- 2h + 2 e^{k_0
\epsilon} h -
         k_0 \tilde{y} ) \right]. \\
\end{align*}

Finally, we are left with the case where $f \neq 0.$ This case leads
to relatively complicated or too long formulas, and we omit it here.

\section{Classification of low-dimensional Symmetry
algebras}\label{s:classification} Group-invariant solutions
corresponding to two subgroups which are conjugate under the adjoint
action of the symmetry group lie on the same orbit. Therefore, in
order to know all the similarity solutions of the ZK equation
invariant under $s$-parameter symmetry subgroups, it is sufficient
to have a classification of $s$-parameter subalgebras, and an
expression for the adjoint representation $\Ad$ of $G.$ In the case
of the ZK equation, as explained in the next section, we shall only
need a classification of $s$-parameter subalgebras under the
corresponding adjoint action $\ad$ of the Lie algebra $L$ of $G$ on
itself, where $1\leq s <3,$ and where $\ad$ denotes the differential
of the $\Ad.$ By the connectedness of $G,$ in order to find an
expression for $\Ad,$ we only need to find it for the flow $\Ad(\exp
(\varepsilon V_i))\cdot \mathbf{v}$ through $\mathbf{v}$ of the
one-parameter subgroup $\exp(\varepsilon V_i)$ generated by each
basis element $V_i$ of $L.$ However, if we denote by $V_{\ad}$ the
infinitesimal generator of $\Ad$ corresponding to each $V \in L,$
then $V_{\ad}$ coincides with the Lie bracket on $L\colon\; V_{\ad}
\cdot \mathbf{v}= [\mathbf{v},V]$ for all $\mathbf{v} \in L.$ We can
therefore reconstruct the adjoint representation  of the group from
that of its Lie algebra, and formulas for $\Ad(\exp (\varepsilon
V))\cdot \mathbf{v}$ based on Lie series or on properties of the
flow exist \cite[Page 205]{olv1}.\par

For the purpose of finding a representative list of all similarity
solutions of the ZK equation, we shall only need as indicated  to
classify one- and two-dimensional subalgebras. Techniques for
performing these classifications can be found in
\cite{ovsy1,pav-crm93, olv1}. They generally consist in mapping or
'reducing' a system of Lie subalgebra generators to an equivalent
system until a canonical representative is achieved. For the ZK
symmetry algebra, we shall need the following set of lemmas which
are essentially based of the commutation relations
~\eqref{eq:com}.\par

We shall often denote collectively by $\mathbf{w}_q,$ where $q$ is
some function of time,  all of the three families of generators
$\mathbf{x}_g, \mathbf{y}_h$ and $\mathbf{z}_f$ in
~\eqref{eq:gtors}. It is clear that $\mathbf{w}_q$ is a linear
function of its argument $q.$  When $q$ is a constant function
represented by its numeric value $q$, $\mathbf{w}_q$ will be denoted
by $\mathbf{w}_{(q)}.$  We shall also often represent a general
vector of the form ~\eqref{eq:genvect} simply by its components in
terms of $k_0,$ $\mathbf{x}_g$, $\mathbf{y}_h$ and $\mathbf{z}_f.$
\begin{equation}\label{eq:genv}
\mathbf{v} =  k_0 \mathbf{v}_0 + \mathbf{x}_g + \mathbf{y}_h +
\mathbf{z}_f  \equiv \set{k_0, g, h, f}
\end{equation}

\begin{lem} \label{le:rule1}
Let $f, g$ and $h$ be given functions of $t,$  and $k_0$ a given
constant. Denote by $\alpha$ and $\beta$ some numbers and by $G$ and
$H$ some functions of $t.$
\begin{enumerate}
\item[(a)]  For some $\alpha$ and $G,$ $\mathbf{v}_0 + \mathbf{x}_g $ is mapped to $\mathbf{v}_0$ under $\Ad(\exp(\alpha\,\mathbf{x}_G)).$
\item[(b)]  For some $\alpha, \beta, G$ and $H$, $\mathbf{v}_0 + \mathbf{x}_g + \mathbf{y}_h$ is mapped to
            $\mathbf{v}_0$ under\\ $\Ad(\exp(\alpha\,\mathbf{x}_G)) \Ad(\exp(\beta\, \mathbf{y}_H)).$
\item[(c)]  For some $\alpha, \beta, G$ and $H$,
            $k_0 \mathbf{v}_0 + \mathbf{x}_g + \mathbf{y}_h + \mathbf{z}_f$ is mapped to $k_0 \mathbf{v}_0 + \mathbf{z}_f$
            under $\Ad(\exp(\alpha\,\mathbf{x}_G)) \Ad(\exp(\beta\, \mathbf{y}_H)),$ provided that $k_0^2 + f^2 \neq 0.$
\item[(d)]  For some $ \beta$ and $H$, $ \mathbf{x}_g + \mathbf{y}_h$ is mapped to
            either $\mathbf{x}_g$ or $\mathbf{y}_h$ under  $\Ad(\exp(\beta\, \mathbf{y}_H))$
\item[(e)]  For some $\alpha$ and $G,$ $\mathbf{v}_0 + \mathbf{x}_g + \mathbf{y}_h$ is mapped to $\mathbf{v}_0+ \mathbf{y}_h$ under
            $\Ad(\exp(\alpha\,\mathbf{x}_G)).$
\item[(f)]For some $\alpha$ and $G,$ $k_0 \mathbf{v}_0 + \mathbf{x}_g + \mathbf{y}_h + \mathbf{z}_f$ is mapped to $k_0 \mathbf{v}_0+ \mathbf{y}_h + \mathbf{z}_f$
            under $\Ad(\exp(\alpha\, \mathbf{x}_G)),$ provided that $k_0^2 + f^2 \neq 0.$
\end{enumerate}
\end{lem}

\begin{proof}
We have $\Ad(\exp(\alpha\, \mathbf{x}_G)) (\mathbf{v}_0)=
\mathbf{v}_0 - 2 \alpha \, \mathbf{x}_G,$  and $\Ad(\exp(\alpha
\mathbf{x}_G)) (\mathbf{v}_0+ \mathbf{x}_g)= \mathbf{v}_0 - 2 \alpha
\mathbf{x}_G + \mathbf{x}_g,$ which has component $\set{1, g- 2
\alpha \,G, 0,0}$. It suffices then to choose $G= g/(2 \alpha)$ with
$\alpha \neq 0,$  which proves part (a). For part (b), we noticed
that when $h$ is zero, the statement reduces to that of part (a).
Otherwise, $V=\mathbf{v}_0 + \mathbf{x}_g + \mathbf{y}_h$ is
transformed under $\Ad(\exp(\beta\, \mathbf{y}_H))$ to a vector with
component $\set{1, g + (\beta /2)(-H h' + h H'), h - \beta H, 0}.$
If we therefore choose $H= h/ \beta$  to cancel the term
$\mathbf{y}_h$ in $V,$ the result will follow from that of part (a).
For (c), we note that if $f=0,$ the problem reduces to that of part
$(b).$ Thus we first assume that $f\neq 0,$ and $k_0=0.$ Then
$\Ad(\exp(\beta\, \mathbf{y}_H))$ will transform $V= \mathbf{x}_g+
\mathbf{y}_h + \mathbf{z}_f$ into a vector $V^{(1)}$ of the form
$\set{0, g_1, h + \beta(- \frac{2}{3}H f'+ f H' ), f},$ where $g_1$
is a certain function of $t.$ Thus to cancel the term $\mathbf{y}_h$
in $V,$ it suffices to choose $H$ as a solution to the equation $h +
\beta(- \frac{2}{3}H f'+ f H' )=0.$ Notice that this is a linear
first order differential equation of the general form
\begin{equation}\label{eq:l1ode}
K'(t)= Q(t, K(t)),
\end{equation}
where $Q$  is a linear function of the unknown function $K,$ and
such linear equations always have a solution. For instance, the
latter equation for $H$ has solution
$$H(t)= C f(t)^{2/3} + f(t)^{2/3}\int_1^t -\frac{h(s)}{\beta f(s)}
ds, \qquad (\beta \neq 0),
$$
where $C$ is an arbitrary constant. Now suppose that $V^{(1)}$
reduces to a vector of the form $\set{0, g^0_1, 0, f}$ when $H$ is
so chosen that the third component of $V^1$ vanishes. Then
$\Ad(\exp(\alpha\, \mathbf{x}_G))$ will map the resulting vector
$V^{(1)}$ to
$$V^{(2)}= \set{0, g- \alpha(- (1/3) G f' + f G'),
0, f}.$$
The condition that the second component of $V^{(2)}$ vanishes is a
linear equation in $G$ of the form ~\eqref{eq:l1ode}, showing that
we can reduce $V$ to $\mathbf{x}_f= k_0 \mathbf{v}_0+ \mathbf{x}_f.$
If on the other hand both $k_0$ and $f$ are nonzero, then
$\Ad(\exp(\beta\, \mathbf{y}_H))$ will map $V= k_0 \mathbf{v}_0 +
\mathbf{x}_g + \mathbf{y}_h + \mathbf{z}_f$ to a vector $V^{(1)}$
with component $\set{k_0, g_1, h+\beta\left(-\frac{H}{3}(3+ 2 f')+ f
H' \right), f},$ for some function $g_1$ depending on $H.$ Let $H_0$
be a function for which the third component of $V^{(1)}$ vanishes,
and denote again by $V^{(1)}= \set{k_0, g_1^0, 0, f}$  the
corresponding image of $V$ under $\Ad(\exp(\beta\,
\mathbf{y}_{H_0})).$ Then,
$$\Ad(\exp(\alpha\, \mathbf{x}_G)) (V^{(1)})= \set{k_0, g- (1/3) \alpha (G(6 + f')- 3 f
G'), f},$$
showing that we can map $V$ as indicated to $k_0 \mathbf{v}_0 +
\mathbf{z}_f,$ and this proves (c).\par For (d), if $h=0,$ then $V=
\mathbf{x}_g + \mathbf{y}_h$ reduces to $\mathbf{x}_g$ and we choose
$\beta=0,$ otherwise it suffices to choose $H$ as a solution to the
linear equation
$$ g + \frac{1}{2}\beta \left(-H h' + h H'\right)=0 ,$$
in order to map $V$ to $\mathbf{y}_h.$ For (e), it suffices to
choose $G= g/ (2 \alpha),$ with $\alpha \neq 0.$ To prove (f), we
simply note that if $k_0=0,$ we let $G$ be the any solution to the
linear {\sc ode}
$$g- \alpha\left( \frac{1}{3} G f' + f G'\right) =0, \qquad (\alpha \neq 0).
$$
If $f=0,$ we let $G$ be a solution to $g- 2 \alpha G=0,$  and if
both $k_0$ and $f$ are nonzero, we let $G$ be the solution to the
linear {\sc ode}
$$ g - \alpha \left(\frac{1}{3} G(6+ f')- 3 f G'\right)=0, \qquad (\alpha \neq
0).
$$
 This completes the proof of Lemma ~\ref{le:rule1}.
\end{proof}
\begin{lem}\label{le:rule2}

Let $g, h$ and $f$ be given functions of time. Then for any
$\varepsilon \in \R,$ and for every nonzero function $K=K(t)$ , we
have
\begin{enumerate}
\item[(a)]
\begin{equation}\label{eq:zf}
\begin{split}
&\Ad(\exp(\varepsilon\, \mathbf{z}_{K}))(\mathbf{x}_g)=
\mathbf{x}_{G(t, \varepsilon)}, \qquad \Ad(\exp(\varepsilon
\mathbf{z}_{K}))(\mathbf{y}_h)= \mathbf{y}_{H(t, \varepsilon)},\\
&\Ad(\exp(\varepsilon\, \mathbf{z}_{K}))(\mathbf{z}_f)=
\mathbf{z}_{F(t, \varepsilon)},
\end{split}
\end{equation}
where the functions $G(t,\varepsilon), H(t, \varepsilon)$ and $F(t,
\varepsilon)$ are given by
\begin{subequations}\label{eq:zfcd}
\begin{align}
G(t, \varepsilon)&= K(t)^{1/3} Q_1\left(\varepsilon - \int_1^t
\frac{ds}{K(s)}\right), \text{
and $Q_1$ satisfies  }               \label{eq:zfxga}\\
g(t)& = K(t)^{1/3} Q_1\left(- \int_1^t \frac{ds}{K(s)}\right).  \label{eq:zfxga}\\
H(t, \varepsilon)&= K(t)^{2/3} Q_2\left(\varepsilon - \int_1^t
\frac{ds}{K(s)}\right), \text{
and $Q_2$ satisfies  } \\
h(t)& = K(t)^{2/3} Q_2\left(- \int_1^t \frac{ds}{K(s)}\right).\\
F(t, \varepsilon)&= K(t) Q_3\left(\varepsilon - \int_1^t
\frac{ds}{K(s)}\right), \text{
and $Q_3$ satisfies  } \\
f(t)& = K(t) Q_3\left(- \int_1^t \frac{ds}{K(s)}\right).
\end{align}
\end{subequations}

\item[(b)]  Let $\mathbf{w}_q,$ where $q= q(t),$ denote collectively all
generators of the form $\mathbf{x}_g, \mathbf{y}_h$ and
$\mathbf{z}_f.$ Then, whenever $q\neq 0,$ there exists a function
$K$ such that $\Ad (\exp(\varepsilon \mathbf{z}_K))
\mathbf{w}_q=\mathbf{w}_{(1)}.$

\end{enumerate}
\end{lem}
\begin{proof}
From Equation ~\eqref{eq:com} we have $[\mathbf{z}_K, \mathbf{x}_g]=
\mathbf{x}_{(Kg'- K'g/3)},$ and this shows that if we express
$\Ad(\exp(\varepsilon\, \mathbf{z}_{K})) (\mathbf{x}_g)$ in terms of
the Lie series and use the linearity of  $\mathbf{x}_g$ as a
function of its argument $g,$ then $\Ad(\exp(\varepsilon
\mathbf{z}_{K})) (\mathbf{x}_g)$ must be of the form
$\mathbf{x}_{G(t, \varepsilon)}$ for a certain function $G= G(t,
\varepsilon).$ Now, the expression for $G(t, \varepsilon)$ follows
from the properties of the flow $\Ad (\exp (\varepsilon
\mathbf{z}_K)) \mathbf{x}_g$ of $\Ad$ through $\mathbf{x}_g$ under
the one-parameter subgroup $\exp(\varepsilon \mathbf{z}_K)$
generated by $K.$ The expressions for $\Ad(\exp(\varepsilon\,
\mathbf{z}_{K})) (\mathbf{y}_h)$ and $\Ad(\exp(\varepsilon\,
\mathbf{z}_K)) (\mathbf{z}_f)$ and for the corresponding functions
$H(t, \varepsilon)$ and $F(t,\varepsilon)$ are derived in a similar
way, and this proves (a).\par
Part (b) simply says that if for instance we have $g \neq 0,$ then
we can choose the function $K(t)$ in such a way that the resulting
function $G(t,\varepsilon)$ in ~\eqref{eq:zf} is $1,$ and that the
same holds for both $h$ and $f$ and for the corresponding functions
$H(t, \varepsilon)$ and $F(t, \varepsilon)$ in ~\eqref{eq:zf}. The
existence of such functions $K$ is just a consequence of a result of
Neuman \cite{neuman}. A proof of the existence of $K$ is also given
in \cite{pavkp}. This completes the proof of the lemma.
\end{proof}

\subsection{Classification of one-dimensional subalgebras}
To implement this classification, we denote as usual by $\mathbf{v}=
k_0 \mathbf{v}_0 + \mathbf{x}_g + \mathbf{y}_h + \mathbf{z}_f$ a
general nonzero vector in $L,$  and for $\mathbf{v}_1, \mathbf{v}_2
\in L $ we write $\mathbf{v}_1 \sim \mathbf{v}_2$ if $\mathbf{v}_2=
\Ad_B (\mathbf{v}_1),$ for some group element $B$ in $G.$ When both
$k_0$ and $f$ equal zero, we have $\mathbf{v} \sim \mathbf{x}_g$ or
$\mathbf{v} \sim \mathbf{y}_h$ by part (e) of Lemma ~\ref{le:rule1}.
Otherwise, we have $k_0^2 + f^2 \neq 0,$ and hence $\mathbf{v} \sim
k_0 \mathbf{v}_0  + \mathbf{z}_f$ by part (c) of Lemma
~\ref{le:rule1}. Thus every one-dimensional subalgebra of $L$ is
equivalent under the adjoint representation to either
$\mathbf{x}_g$, $\mathbf{y}_h$ or $k_0 \mathbf{v}_0 + \mathbf{z}_f.$
Since we have $\Ad_B (k_0 \mathbf{v}_0 + \mathbf{z}_f)= k_0
\mathbf{v}_0 + Ad_B (\mathbf{z}_f),$ for every group element $B,$ it
follows from Lemma ~\ref{le:rule2} that if $f=0,$ then $(k_0
\mathbf{v}_0 + \mathbf{z}_f) \sim \mathbf{v}_0,$ other wise $(k_0
\mathbf{v}_0 + \mathbf{z}_f) \sim k_0 \mathbf{v}_0 +
\mathbf{z}_{(1)}.$ The same lemma thus implies that every one
dimensional subalgebra of the ZK symmetry algebra is equivalent
under the adjoint representation to one of the following Lie
algebras

\begin{equation}\label{eq:d1salg}
\mathcal{L}_{1,1}=\set{\mathbf{x}_{(1)}},\quad
\mathcal{L}_{1,2}=\set{\mathbf{y}_{(1)}},\quad
\mathcal{L}_{1,3}=\set{\mathbf{v}_0}\quad \text{ or }\quad
\mathcal{L}_{1,4}=\set{k_0 \mathbf{v}_0 + \mathbf{z}_{(1)}}.
\end{equation}

Note that according to ~\eqref{eq:gtors}, we have
\begin{equation}\label{eq:zto1}
\mathbf{x}_{(1)}= \pd_x, \qquad \mathbf{y}_{(1)}= \pd_y,\quad
\text{and }\quad \mathbf{z}_{(1)}= \pd_t,
\end{equation}
which shows that the canonical forms thus obtained are very
simplified. To see why these four Lie algebras are non equivalent,
we first note that $\Ad(\exp(\varepsilon\, \mathbf{v}_0))$ acts only
diagonally, by scalling its argument, while for every function
$F=F(t),$ $\Ad(\exp(\varepsilon\, \mathbf{z}_F))$ leaves
$\mathbf{v}_0$ invariant and maps, by Lemma ~\ref{le:rule2},
$\mathbf{x}_g, \mathbf{y}_h$ and $\mathbf{z}_f$ to
$\mathbf{x}_\frak{g}, \mathbf{y}_\frak{h}, \mathbf{z}_\frak{f},$
respectively, for some functions  $\frak{g}= \frak{g}(t,
\varepsilon), \frak{h}=\frak{h}(t, \varepsilon)$ and
$\frak{f}=\frak{f}(t, \varepsilon).$ On the other hand, if we set
$$\mathcal{F}= \Ad(\exp(\alpha\, \mathbf{x}_G)) \Ad(\exp(\beta\,
\mathbf{y}_H)),$$
then we have,
\begin{align*}
\mathcal{F}(\mathbf{v}_0)&= \set{1, -2 \alpha G,-\beta H, 0 } \sim \mathbf{v}_0\\
\intertext{and}
\mathcal{F}(k_0 \mathbf{v}_0 + \mathbf{z}_{(1)})&= \set{k_0
\mathbf{v}_0, \alpha(- 2 G + G')+
(1/4) \beta^2 (H'^2 - H H''), \beta (- H+ H'), 1 }\\
           & \sim
k_0\mathbf{v}_0 + \mathbf{z}_{(1)}.
\end{align*}
We also have $\mathcal{F}(\mathbf{x}_{(1)})= \mathbf{x}_{(1)}$, and
$\mathcal{F}(\mathbf{y}_{(1)})= \mathbf{y}_{(1)}.$ This shows that
the four Lie algebras thus obtained are non-equivalent under the
adjoint representation of $G,$ and thus completes the classification
of one-dimensional subalgebras of $L.$
\subsection{Classification of two-dimensional subalgebras}
Every two dimensional subalgebra $L_2$ of $L$ is solvable and if we
denote by $\mathcal{B} =\set{V_1, V_2}$ a basis of $L_2,$ then
either $L_2$ is abelian and thus $[V_1, V_2]= 0,$ or $L_2$ has a
nonzero nilradical, in which case its commutation relations can be
put in the form $[V_1, V_2]= V_1.$ To classify two dimensional
subalgebras of $L$, we let $V_1$ be in one of the canonical forms
~\eqref{eq:d1salg}, while $V_2= \mathbf{v}$ is a general vector of
the form ~\eqref{eq:genv}. By the possible forms for $V_1$ given in
~\eqref{eq:d1salg}, we shall therefore have to consider four basic
cases for $\mathcal{B}$. To reduce $L_2$ to a canonical form, we
transform it under maps of the form $\Ad_B,$ for appropriately
chosen group elements $B$ in the symmetry group $G,$ and also make
use of the commutation relations in $L_2$ to obtain some
restrictions on the functions $g, h, f$ and the free parameter $k_0$
defining $V_2 = \mathbf{v}.$ If $\mathcal{B}_1$ is a basis of a
subalgebra $S_1$ and $\mathcal{B}_2$ a basis of a subalgebra $S_2,$
we write $\mathcal{B}_1 \sim \mathcal{B}_2$ if $S_2= \Ad_B (S_1)$
for some $B \in G.$ In this subsection, we shall let $c_1, c_2$ and
$c_3$ denote three arbitrary constants.
\subsubsection{Abelian subalgebras}
\begin{flushleft}
Case (a):  $\quad \mathcal{B}= \set{\mathbf{v}_0, \mathbf{v}}.$\\[2mm]
\end{flushleft}

The condition $[\mathbf{v}_0, \mathbf{v}]=0,$ implies that $g=h=0.$
Thus $\mathbf{v}= k_0 \mathbf{v}_0 + \mathbf{z}_f.$ Since for every
$\varepsilon \in \R$ and every function $K$ we have
$$\Ad(\exp(\varepsilon\, \mathbf{z}_K)) (k_0 \mathbf{v}_0 +
\mathbf{z}_f)= k_0 \mathbf{v}_0+\Ad(\exp(\varepsilon\,
\mathbf{z}_K)) (\mathbf{z}_f),$$
 and since $\Ad(\exp(\varepsilon
\mathbf{z}_K )) \mathbf{v}_0= \mathbf{v}_0,$ it follows that if
$f=0,$ then $\mathbf{v} \sim \mathbf{v}_0.$ Otherwise, $\mathbf{v}
\sim k_0 \mathbf{v}_0 + \mathbf{z}_{(1)}.$ Since $L_2$ is
two-dimensional, we must have
$\mathcal{B}\sim \set{\mathbf{v}_0, k_0 \mathbf{v}_0 + \mathbf{z}_{(1)}}.$\\[1mm]
\begin{flushleft}
Case (b):  $\quad \mathcal{B}= \set{ k_1 \mathbf{v}_0 + \mathbf{z}_{(1)},
\mathbf{v}}, \; k_1 \in \R.$\\[2mm]
\end{flushleft}
 We must
have in this case
\begin{equation}\label{eq:ab2}
g= c_1 e^{2 k_1 t}, \qquad h= c_2 e^{k_1 t}, \quad \text{and } f=
c_3.
\end{equation}
Thus $\mathcal{B} \sim \set{k_1 \mathbf{v}_0 + \mathbf{z}_{(1)},
\mathbf{v}},$ where the components $g, h, f$ of $\mathbf{v}$ are
given by ~\eqref{eq:ab2}.

\begin{flushleft}
Case (c):  $\quad \mathcal{B}= \set{\mathbf{x}_{(1)}, \mathbf{v}}.$\\[2mm]
\end{flushleft}

We must have $f= 6 k_0 t+ c_3.$ If $k_0=0$ and $c_3=0,$ then by
 part (d) of Lemma ~\ref{le:rule1}, under $\Ad(\exp(\beta\, \mathbf{y}_H))$,  we have $\mathbf{v} \sim \mathbf{x}_g$
 or $\mathbf{v} \sim \mathbf{y}_h,$ and thus, since $\Ad(\exp(\beta\, \mathbf{y}_H)) \mathbf{x}_g= \mathbf{x}_g$ for all $g,$
 we have $\mathcal{B} \sim \set{\mathbf{x}_{(1)}, \mathbf{x}_g}$ or $\mathcal{B} \sim \set{\mathbf{x}_{(1)},
 \mathbf{y}_h}.$ If $k_0 \neq 0$ or $c_3 \neq 0,$ then by part (c) of Lemma
 ~\ref{le:rule1}, we have $\mathbf{v} \sim k_0 \mathbf{v}_0 + \mathbf{z}_{(6 k_0 t+ c_3)},$ and hence
 $$\mathcal{B} \sim \set{\mathbf{x}_{(1)},k_0 \mathbf{v}_0 + \mathbf{z}_{(6 k_0 t+ c_3)} }.$$

\begin{flushleft}
Case (d):  $\quad \mathcal{B}= \set{\mathbf{y}_{(1)}, \mathbf{v}}.$\\[2mm]
\end{flushleft}

We must have $f= -(3/2) k_0 t + c_3,$ and $h=c_2.$ Thus if $k_0=0$
and $c_3=0,$ then $\mathbf{v} \sim \mathbf{x}_g +
\mathbf{y}_{(c_2)},$ and no further reduction of the basis
$\mathcal{B} = \set{\mathbf{y}_{(1)}, \mathbf{x}_g +
\mathbf{y}_{(c_2)}}$ is possible. If $k_0\neq 0$ or $c_3\neq 0,$
then
$$\mathcal{B} \sim \set{\mathbf{y}_{(1)}, k_0 \mathbf{v}_0 +
\mathbf{y}_{(c_{\, 2})} + \mathbf{z}_{(-(3/2) k_0 t + c_3)}}. $$
This is because by part (f) of Lemma ~\ref{le:rule1}, we have
$\mathbf{v} \sim k_0 \mathbf{v}_0 + \mathbf{y}_{c_2} +
\mathbf{z}_f,$ under $\Ad(\exp(\alpha\, \mathbf{x}_G)),$ and the
operator $\Ad(\exp(\alpha\, \mathbf{x}_G))$ leaves
$\mathbf{y}_{(1)}$ invariant for every $\alpha$ and $G.$ Again, all
the canonical representatives of two-dimensional subalgebras thus
obtained are clearly pairwise nonequivalent, and this completes the
classification problem in the abelian case.
\subsubsection{Non-abelian subalgebras}\label{ss:nonab}
\begin{flushleft}
Case (a):  $\quad \mathcal{B}= \set{\mathbf{v}_0, \mathbf{v}}$ or
$ \mathcal{B}= \set{k_1 \mathbf{v}_0 +
\mathbf{z}_{(1)}, \mathbf{v}}, \quad k_1 \in \R.$\\[2mm]
\end{flushleft}

Since $\mathbf{v}_0$ is not in the derived subalgebra of $L,$ the
case $\ \mathcal{B}= \set{\mathbf{v}_0, \mathbf{v}}$ cannot occur,
while the case $\mathcal{B}= \set{k_1 \mathbf{v}_0 +
\mathbf{z}_{(1)}, \mathbf{v}}$ occurs only if $k_1=0,$  in which
case we must have
$$ g= c_1,\quad  h= c_2,\quad \text{ and} \quad  f= t+ c_3
$$
Consequently, in this case
$$ \mathcal{B} \sim \set{\mathbf{z}_{(1)}, k_0 \mathbf{v}_0 +
\mathbf{x}_{(c_1)}+ \mathbf{y}_{(c_2)} + \mathbf{z}_{(t+ c_3)}}.$$

\begin{flushleft}
Case (b):  $\quad \mathcal{B}= \set{\mathbf{x}_{(1)}, \mathbf{v}}.$\\[2mm]
\end{flushleft}
We must have in this case $f= (3- 6 k_0)t + c_3.$ Thus if $k_0= 1/2$
and $c_3=0,$ then by Lemma ~\ref{le:rule1}, part (b),
$\Ad(\exp(\alpha\, \mathbf{x}_G)) \Ad(\exp(\beta\, \mathbf{y}_H))$
will map $\mathbf{v}$ to $\mathbf{v}_0$ and leave $\mathbf{x}_{(1)}$
unchanged. Hence, $\mathcal{B} \sim \set{\mathbf{x}_{(1)},
\mathbf{v}_0}$ in this case. If either $k_0 \neq 1/2$ or $c_3 \neq
0,$ the same lemma implies that $\mathbf{v} \sim k_0 \mathbf{v}_0 +
\mathbf{z}_f,$ under a similar transformation that leaves
$\mathbf{x}_{(1)}$ unchanged. Consequently,  in this second case we
have $\mathcal{B} \sim \set{\mathbf{x}_{(1)}, k_0 \mathbf{v}_0 +
\mathbf{z}_{((3-6 k_0)t+ c_3)}}.$

\begin{flushleft}
Case (c):  $\quad \mathcal{B}= \set{\mathbf{y}_{(1)}, \mathbf{v}}.$\\[2mm]
\end{flushleft}
The commutation relations imply that $f= (3/2)(1- k_0)t + c_3,$ and
$h= c_2.$ Thus if $k_0=1$ and $c_3=0,$ by part (e) of Lemma
~\ref{le:rule1}, $\Ad(\exp(\alpha\, \mathbf{x}_G))$ maps
$\mathbf{v}$ to $\mathbf{v}_0 + \mathbf{y}_{(c_2)}.$ Since this
operator leaves $\mathbf{y}_{(1)}$ invariant, we must have
$\mathcal{B} \sim \set{\mathbf{y}_{(1)}, \mathbf{v}_0 +
\mathbf{y}_{(c_2)}}$ in this case.  If $k_0 \neq 1$ or $c_3 \neq 0,$
part (f) of the same lemma shows that $\Ad(\exp(\alpha\,
\mathbf{x}_G))$ maps $\mathbf{v}$ to $k_0 \mathbf{v}_0 +
\mathbf{y}_{(c_2)} + \mathbf{z}_f.$ Consequently,
$$\mathcal{B}
\sim \set{\mathbf{y}_{(1)}, k_0 \mathbf{v}_0 + \mathbf{y}_{(c_2)} +
\mathbf{z}_{( (3/2)(1- k_0)t + c_3 )}}.$$ Since the canonical forms
thus obtained are non-equivalent by construction, this completes the
classification of two-dimensional subalgebras of $L$ in the
non-abelian case. \par
We have thus obtained the following list of canonical forms of
non-equivalent two-dimensional subalgebras of $L,$ in which $k_0,
k_1, c_1, c_2,$ and $c_3$ are free parameters, unless otherwise
specified.
\begin{enumerate}
\topsep = 2mm
\itemsep = 2.5mm
\item[(1)] {$\mathsf{ Abelian\; subalgebras}$}
\begin{enumerate}
\topsep = 2.5mm
\itemsep = 2mm
\item[$\mathcal{L}_{2,1}$] = $\set{{\mathbf{v}_0, k_0 \mathbf{v}_0 + \mathbf{z}_{(1)}}}$
\item[$\mathcal{L}_{2,2}$] = $\set{k_1 \mathbf{v}_0 + \mathbf{z}_{(1)}, k_0 \mathbf{v}_0 + \mathbf{x}_{(c_1 e^{2 k_1 t})} + \mathbf{y}_{(c_2 e^{k_1 t})}+
\mathbf{z}_{(c_3)}}$
\item[$\mathcal{L}_{2,3}$] = $\set{\mathbf{x}_{(1)}, \mathbf{x}_{g}}, \qquad (g' \neq 0)$
\item[$\mathcal{L}_{2,4}$] = $\set{\mathbf{x}_{(1)}, \mathbf{y}_{h}  } $
\item[$\mathcal{L}_{2,5}$] = $\set{\mathbf{x}_{(1)}, k_0 \mathbf{v}_0 + \mathbf{z}_{(6 k_0t+ c_3)}}, \qquad k_0^2+ c_3^2 \neq 0$
\item[$\mathcal{L}_{2,6}$] = $\set{\mathbf{y}_{(1)},  \mathbf{x}_{g}+ \mathbf{y}_{(c_2)}  }$
\item[$\mathcal{L}_{2,7}$] = $\set{\mathbf{y}_{(1)}, k_0 \mathbf{v}_0+ \mathbf{y}_{(c_2)}+ \mathbf{z}_{\left((-3/2)(k_0 t + c_3)\right)}}, \qquad k_0^2 + c_3^2 \neq 0$
\end{enumerate}
\item[(2)] {$\mathsf{ Non\!-\!abelian\; subalgebras}$}
\begin{enumerate}
\topsep = 2mm
\itemsep = 1.5mm
\item[$\mathcal{L}_{2,8}$] = $\set{\mathbf{z}_{(1)}, k_0 \mathbf{v}_0 + \mathbf{x}_{(c_1)} + \mathbf{y}_{(c_2)} + \mathbf{z}_{(t+ c_3)}}$
\item[$\mathcal{L}_{2,9}$] = $\set{\mathbf{x}_{(1)}, \mathbf{v}_0}$
\item[$\mathcal{L}_{2,10}$] = $\set{\mathbf{x}_{(1)}, k_0 \mathbf{v}_0  + \mathbf{z}_{((3- 6 k_0)t +
c_3)}},  \qquad (k_0 - \frac{1}{2})^2 + c_3^2 \neq 0 $
\item[$\mathcal{L}_{2,11}$] = $\set{\mathbf{y}_{(1)}, \mathbf{v}_0 + \mathbf{y}_{(c_2)}}$
\item[$\mathcal{L}_{2,12}$] = $\set{\mathbf{y}_{(1)}, k_0 \mathbf{v}_0 + \mathbf{y}_{(c_2)}+ \mathbf{z}_{((3/2)(1-k_0)t + c_3)}},\qquad  (k_0-1)^2+ c_3^2 \neq 0$
\end{enumerate}
\end{enumerate}

\section{Similarity reductions of the ZK
equation}\label{s:reduction} We turn our attention in this section
to the problem of finding the group-invariant solutions of the ZK
equation. By group invariant solutions we mean solutions which are
invariant by the group transformations in the sense that, roughly
speaking, each of them is transformed into itself by every group
transformation. In other words, their graph is a locally
$G$-invariant subset. In the case of a {\sc pde} such solutions can
be found by solving a differential equation in fewer independent
variables.  We shall therefore seek a reduction of the ZK equation
to $(1+1)$-dimensional equations or to ordinary differential
equations. To reduce the equation to one with $s$ fewer independent
variables, the general procedure \cite{olv1} is to look for
subgroups of the full symmetry group whose orbits have dimension
$s.$ Each such subgroup yields a set of $3-s$ functionally
independent invariants which can be written in the form $\xi=
\eta(\mathbf{w}),$ where $\mathbf{w} = (t,x,y),$ in addition to
another functionally independent invariant of the form $w= \zeta
(\mathbf{w}, u).$ We must clearly have $1\leq s <3,$ since the ZK
equation has exactly three independent variables.  Solving this last
equality for the solution $u$ of the original equation shows that
$u$ will always be determined in the case of the ZK equation by some
equations of the form
\beq \label{eq:u} u(t,x,y)= \alpha + \beta w(\xi), \qquad \xi=
\eta(\mathbf{w}) \eeq
where $\alpha$ and $\beta$ are functions of $\mathbf{w}.$ Equation
~\eqref{eq:u} can be used to rewrite the original equation in terms
of the invariant functions $\xi= \eta(\mathbf{w}),$ and $w$
considered as new variables, and this yields the reduced system in
only $3-s$ independent variables \cite{olv1}. Expressing the
solution of the reduced equation in the form $w= F(\xi),$ and
substituting in this last equality the expressions for $w$ and $\xi$
from equation ~\eqref{eq:u}, the solution $u$ of the original
equation can be found.  We treat separately reductions by one
dimensional subgroups and reductions by two dimensional subgroups.

\subsection{Reduction by one-dimensional subgroups}
\label{sb:1sbgp}
 As already indicated, a reduction by one dimensional subgroups will
 yield a partial differential equation in only $2= 3-1$ independent
 variables. The generic form of the infinitesimal generator of such a group is
 given by  ~\eqref{eq:genvect}.
In the case of a reduction by one-dimensional subgroups, the
function the $ \xi= \eta(\mathbf{w})$ in ~\eqref{eq:u} will have the
form $\xi=(\eta_1, \eta_2),$ and for simplicity of notation  we
shall set $r=\eta_1$ and $z=\eta_2.$ In order to completely define a
reduced equation, we will only need  to give the explicit formulas
for $u,z,$ and $r$  corresponding to Equation ~\eqref{eq:u} and the
reduced equation itself. We first give reductions based on the
classification of one-dimensional subalgebras given in
~\eqref{eq:d1salg}, and then reductions based of a direct case
analysis afterward.
\par
\subsubsection{Reductions based on one-dimensional
subalgebras classification }\label{ss:byclas} This procedure
consists in finding the reduced equation corresponding to each of
the four non-equivalent canonical forms of one-dimensional
subalgebras of $L$ given in ~\eqref{eq:d1salg}. The list of all
solutions to the four resulting reduced equations should represent
an optimal list of $1$-parameter group invariant solutions, with the
property that every other $1$-parameter group invariant solution of
the ZK equation can be mapped to precisely one solution in the list
via the adjoint representation of $G.$\par
\begin{enumerate}
\topsep = 1.5mm
\itemsep = 1.5mm

\item[(1)] Reduction by $\mathcal{L}_{1,1}= \set{\mathbf{x}_{(1)}}.$ \\[2mm]
For $\mathbf{x}_{(1)}= \pd_x,$ the reduction formula is simply $u=
u(t,y), r=t,$ and $z=y.$ The reduced equation is the linear equation
$$ u_{y,y}=0$$
with solution $u= q_1 y + q_2$, where $q_1$ and $q_2$ are arbitrary
functions of $t$.
\item[(2)] Reduction by $\mathcal{L}_{1,2}= \set{\mathbf{y}_{(1)}}.$ \\[2mm]
For $\mathbf{y}_{(1)}= \pd_y,$ the reduction formula is $u=u(t,x),\,
r=t,\, z=x.$ The reduced equation is
$$- u_x^2 + u_{t,x}- u u_{x,x}=0. $$
\item[(3)] Reduction by $\mathcal{L}_{1,3}= \set{\mathbf{v}_0}.$ \\[2mm]
The reduction formula is
\begin{equation}\label{eq:l13formula}
    u= x w(t,z), \quad r= t, \quad z=y^2/2,
\end{equation}
and the corresponding reduced equation is
\begin{equation}\label{eq:redl13}
w_r - 2 w_z - (w- z w_z)^2 -z w_{r,z} - 4z w_{z,z} - z^2 w
w_{z,z}=0.
\end{equation}

\item[(4)] Reduction by $\mathcal{L}_{1,4}= \set{k_0 \mathbf{v}_0 + \mathbf{z}_{(1)}}.$ \\[2mm]
The reduction formula is
$$ u = x w(r,z), \quad r= 2 k_0 t - \ln(x), \quad z= y^2/x ,$$ which
gives rise to the reduced equation
\begin{equation}\label{eq:l14}
\begin{split}
 2 w_z &+ (-w + w_r + z w_z)^2 + 2 k_0 (w_{r,r}+ zw_{r,z}) + 4 z
w_{z,z}\\
& + w (- w_r + w_{r,r}+ z (2 w_{r,z}+ z w_{z,z}))- 2 k_0 w_r=0.
\end{split}
\end{equation}
For $k_0=0,$ this last equation corresponds to the much simpler
reduction by $\set{\mathbf{z}_{(1)}}$ given by
\begin{equation}\label{eq:l14z1}
  u_x^2 + u u_{x,x} + u_{y,y}=0.
\end{equation}
\end{enumerate}
\subsubsection{Reduction by direct case analysis}\label{ss:byanal}This procedure
consists in obtaining all possible reductions of the ZK equation by
one-parameter subgroups, by a suitable consideration of each
relevant case separately. Although it gives directly all possible
cases of reduced equations, it might involve tedious or too long
calculations in certain cases.\par
\underline{\bf Case 1} $\colon$ $g=h=f=0,$ and $k_0 \neq 0$\\
This is just a reduction by the subgroup generated by
$\mathbf{v}_0.$  The reduction formula and corresponding reduced
equations are already given in ~\eqref{eq:l13formula} and
~\eqref{eq:redl13}.\\[2mm]
\underline{\bf Case 2} $\colon$ $h=f=0,$ and $ g \neq 0,$ and $k_0$ arbitrary\\

$\qquad$ \underline{Case 2a} $\colon$ $k_0=0$\\
In this case we have
\beq \label{eq:frmp2a} u= \frac{g'}{g}(w-x),\quad r=t, \quad z= y.
\eeq
The reduced equation is
\beq \label{eq:redp2a}  g''+ g'w_{z,z}=0. \eeq This gives rise to
the invariant solution
$$
u = \frac{g'}{g}(w-x), \quad w=a(t) y + b(t), \quad g(t)= a_1 t+
a_2$$
where $a(t),\, b(t)$ are arbitrary functions, while $a_1$ and  $a_2$
are arbitrary constants.\\[1mm]

$\qquad$ \underline{Case 2b} : $ k_0\neq 0$

The reduction formula is \beq \label{eq:frmp2b} u= y^2 w(t, z) +
\frac{g'}{2 a},\quad r=t, \quad z= \frac{1}{y^2}(x+ g/(2k_0)), \eeq
and the reduced equation is \beq \label{eq:redp2b} - 2 w+ 2 z w_z -
w_z^2 +  w_{z,t} - (w+ 4 z^2)w_{z,z}=0. \eeq

\underline{\bf Case 3} $\colon$ $f=0,$ and $ h \neq 0$ \\

$\qquad$ \underline{Case 3a} $\colon$ $k_0=0$\\
The reduction formula is
\beq \label{eq:frmp3a} u= \frac{1}{4 h}( w(t, z)- (2 y g'+ y^2
h'')),\quad r=t, \quad z= x- \frac{y}{2 h}\left( g + y h'/2\right).
\eeq
The reduced equation is the system
\begin{align}
   h' w_z + h''&=0 \notag  \\
   w_{z,z} &=0     \notag  \\
   - w_{z}^2 + 4 w_{z,t} - w w_{z,z}&=0. \label{eq:redp3a_1}
\end{align}
This system implies that
$$w = \frac{z}{-t + a_1} + a(t),
\qquad h= \frac{a_2}{2} t^2 + a_1 a_2 t + a_3, $$
and the corresponding invariant solution is
$$u=
\frac{-t (2 a_1 + t)x a_2 - 2 x a_3 + a(t) (a_1 + t) (t (2 a_1
+t)a_2 + 2 a_3)+ y (g + (-a_1 -t)g')} {(a_1+t)(t (2 a_1 +t)a_2 + 2
a_3)},
$$
where $a(t)$ is an arbitrary function, while $a_1, a_2$ and $a_3$
are arbitrary constants. \\[1mm]

$\qquad$ \underline{Case 3b} $\colon$ $k_0 \neq 0$\\

The reduction formula is \\
\begin{align*}
\label{eq:frmp3b}
u &= (2 h + k_0 y)^2 w(t, z) + \frac{k_0 g' + 2(h + k_0 y)h''}{2 k_0^2} \\
r &=t,\quad z = \frac{1}{(2 h + k_0 y)^2} \left(  x+ \frac{k_0 g + 2
(h + k_0 y)h'}{2 k_0^2} \right)
\end{align*}
and the corresponding reduced equation in which we may set $k_0=1$
is
\beq \label{eq:redp3b}   - 2 k_0^2 w + 2 k_0^2 z w_z - w_z^2 +
w_{z,t}- (w+ 4 k_0^2 z^2) w_{z,z}=0. \eeq

\underline{\bf Case 4} : $f\neq 0$ \\

In this case the coefficient of $\partial_t$  in the expression of
the operator $\mathbf{z}_f$ is nonzero and thus we may no longer
consider the variable $t$ as trivially invariant.  We make the
simplifying assumption that in ~\eqref{eq:genvect} we have $k_0 =
0,$ so that we are looking for a reduction of the equation by a
generic vector field of the form
\begin{equation}
\label{eq:genvect234} \mathbf{v}=  \mathbf{x}_g + \mathbf{y}_h+
\mathbf{z}_f.
\end{equation}
The expressions for the invariant functions $w, z,$ and $r$ are
found after some long calculations to be given by
\begin{subequations}\label{eq:invcase4}
\begin{align}
 r &= \frac{1}{f^{2/3}} \left( y - \frac{1}{3} f^{2/3} G
\right)
\\[1.5mm]
\begin{split}
      z &= \frac{1}{54 f^{4/3}} \left[   54 x f + 3 f^{2/3} h G - 9 y (h + y f')
       \right]\\
       &\quad +\frac{1}{54}\left[ -9 \int \left( \frac{g}{f^{4/3}} +
       \frac{G h'}{3 f^{2/3}} + \frac{1}{9} G^{\, 2} f'' \right) dt + G^2 f'
       \right]
\end{split}\\[1.5mm]
w   &=  \frac{36 u f^2 - (h + 2 y f')^2 + 6 f (g + 2 x f'+ y (h' + y
f''))}{36 f^{4/3}},\\
 \intertext{ where}
G & = \int \frac{h}{ f^{5/3}} dt. \notag
\end{align}
\end{subequations}
The reduced equation in this case is
\beq \label{eq:redf} w_{z}^2  + w w_{z, z} + w_{r,\, r}=0. \eeq
It appears that equations ~\eqref{eq:l14z1} and ~\eqref{eq:redf} are
exactly the same, despite the long calculations leading to the
latter equation, and this confirms in some sense the results of part
(c) of Lemma ~\ref{le:rule1} and that of Lemma ~\ref{le:rule2} which
assert  that when $k_0=0,$ and $f\neq 0,$  the corresponding vector
$\mathbf{v}$ of ~\eqref{eq:genvect234} is equivalent to
$\mathbf{z}_{(1)}.$ In a similar way, in the case where $f\neq 0$
and $k_0 \neq 0,$ the equation obtained by direct case analysis,
although not calculated in this paper, should match that obtained in
~\eqref{eq:l14}. Similar correspondences can be established between
the equations obtained in Section ~\ref{ss:byclas} and Section
~\ref{ss:byanal}.\par

 The classification obtained by direct case analysis in this section contains
an exhaustive classification of all possible similarity reductions
by one dimensional subgroups, except the case when $k_0 \neq0$ and
$f\neq 0.$

\subsection{Reduction by two-dimensional subgroups}
\label{sb:2sbgp} Despite the list of all twelve canonical forms of
non-equivalent two-dimensional subalgebras of the ZK symmetry
algebra given in Section ~\ref{ss:nonab}, we shall only consider
reductions of the equation by all pairs of distinct generators of
the symmetry algebra (except those pairs of the form
$\set{\mathbf{v}_{F_1},\mathbf{v}_{F_2} }$ defined by the same type
of generators). This is because space limitations precludes the
treatment of all twelve cases in this paper. Also, as indicated in
\cite{pavkp}, these reductions do not yield additional solutions
when compared with those obtained from reductions by one-dimensional
subgroups. However, they give rise to equations which are normally
easier to solve. Each of the pairs of generators that we shall
consider generates a subalgebra whose corresponding action turns out
to have orbits of dimension two, and thus reduces the equation to an
{\sc ode}. We shall always indicate to which of the classified
canonical forms of two-dimensional subalgebras they correspond.
\par

The similarity variables in terms of which the reduced equations are
expressed will be found by solving a system of first order partial
differential equation of the form
\begin{equation}
\begin{cases}
\label{eq:inv2}
\mathbf{v}_ {F_1} \cdot I = 0 &  \\
\mathbf{w}_{F_2} \cdot I = 0 &
\end{cases}
\end{equation}
where each $\mathbf{v}$ and $\mathbf{w}$ are operators depending on
the functions $F_1$ and $F_2,$ respectively. Because of the
arbitrary functions appearing in ~\eqref{eq:inv2}, it can be
difficult to solve this system directly in certain cases. However,
Equation ~\eqref{eq:inv2} becomes more tractable when it is
transformed into the equivalent adjoint system of total differential
equation. The process for this transformation and the method for
solving the resulting system of total differential equations are
described in \cite{ndog04, forsyth}. \par

One important aspect of Equation ~\eqref{eq:inv2} is its
integrability condition. If we consider for example this system with
 $\mathbf{v}_ {F_1}= \mathbf{x}_g$, and
$\mathbf{w}_ {F_2}= \mathbf{z}_f$ (see Equation ~\eqref{eq:com}) and
$f(t)=1,$ then we have $\mathbf{x}_{g} \equiv (0, g, 0, -g)$  and
$\mathbf{z}_{f} \equiv (1, 0, 0, 0).$ We may thus ignore the
variable $y,$ in which case a function $F(x,u)$ is
 an invariant if and only if $\mathbf{x}_{g} \cdot F(x,u)=0.$
 This in turn is equivalent to
$$
\frac{g'}{g}= \frac{F_x (x,u)}{F_u (x,u)}= a, \qquad \text{where $a$
is a constant}.
$$
This shows that the function $g$ of the operator $\mathbf{x}_g$ most
satisfy $g= C e^{at},$ for some constant $C.$ More generally, for an
arbitrary pair $\set{g,f},$ the compatibility condition for Equation
~\eqref{eq:inv2}  is given by

\begin{align*}
 \qquad &\quad  -3 f g'^2 - g^2 f'' + 3 g (f' g' + f g'')=0 .\\
\end{align*}

This means that for any solution of ~\eqref{eq:inv2} to exist, we
ought to have in this case
\begin{equation}
\label{eq:comp13} g^3 = f \exp \left(\int \frac{K}{f} \, dt \right).
\end{equation}
When the integrability condition is satisfied, Equation
~\eqref{eq:inv2} will yield two invariants $w,$ and $z,$ and in
order to obtain the reduced equation we shall set $w= w(z).$

\begin{flushleft} {\bf Reduction by}  $\set{\mathbf{v}_0, \mathbf{x}_g}$ \end{flushleft}
The corresponding Lie algebra belongs to type $\mathcal{L}_{2,9}.$
  The reduction formula is
$$ u = g' y^2 \left(w(z) - \frac{x}{g y^2} \right), \qquad z= t $$
and the reduced equation is
$$
-2 w g' - \frac{g''}{g}=0.
$$

This yields the group-invariant solution
$$
u= - y^2 g'\left( \frac{x}{y^2 g} + \frac{g''}{2 g g'}  \right).
$$

\begin{flushleft} {\bf Reduction by} $\set{\mathbf{v}_0, \mathbf{y}_h}$ \end{flushleft}
This Lie algebra belongs to the family of type $\mathcal{L}_{2,11}$
Lie algebras. The reduction formula is

$$u= \left[ \left( \frac{2 x}{h'} - \frac{y^2}{2 h}\right) w(z) -
\frac{y^2}{2 h} \right] \frac{h''}{2}, \qquad z=t.$$

The reduced equation is a generalized Riccati equation of the form
\begin{equation}
\label{eq:red02} w' + \alpha(t)w^2 + \beta(t)w + \gamma(t)=0
\end{equation}

where $\alpha(t), \beta(t),$ and  $\gamma(t)$ are some functions of
time.

\begin{flushleft} {\bf Reduction by} $\set{\mathbf{x}_g, \mathbf{y}_h} $   \end{flushleft}
This is a Lie algebra of type  $\mathcal{L}_{2,4}.$  The reduction
formula is
$$
u = \frac{1}{ 2 h g} \left[ w(z) - 2h g' x + (g' h' - g
h'')\frac{y^2}{2} \right], \qquad z=t.
$$
The reduced equation corresponds in this case to a mere
integrability condition  of the form
$$- g'h'+ 2 h g''- g h''=0. $$
A first integral for this last equation is
$$ h'- 2 \frac{g'}{g}h + \frac{K}{g}=0, \qquad K= \text{constant},$$

and this leads to a functional relation between $h$ and $g$ of the
form
$$ h = \alpha g^2 - K g^2 \int g^{-3} \, dt, \quad \alpha \equiv \text{constant}.$$

The corresponding exact solution of ~\eqref{eq:zk} is thus
\begin{equation}\label{eq:sol12} u= \frac{w(t) - g^2 (2 x g'
+ y^2 g'') (\alpha - K \int g^{-3} d t)}{2 g^3 (\alpha - K \int
g^{-3} d t)}
\end{equation}

where $w$ is an arbitrary function of time. For $K=0,$ it reduces to
\begin{equation}
\label{eq:sol12b}
 u= \frac{w(t) - \alpha g^2 (2 x g' + y^2 g'')}{2 \alpha g^3}.
\end{equation}

\begin{flushleft} {\bf Reduction by} $\set{\mathbf{v}_0, \mathbf{z}_f}$ \end{flushleft}
This is a particular case of type $\mathcal{L}_{2,1}$ Lie algebras.
The equivalent system of total differential equations takes in this
case the form
\begin{align*}
dy \; &= \; \frac{ 6 x y f' - y^3 f''}{12 x f}\, dt + \frac{y}{2x}\, dx \\
du \; &= \; \frac{1}{12 xf}\left[- 12 u x f' - 4 x^2 f'' - 2u y^2
f'' - 2 x y^2 f^{(3)}\right] \, dt + \frac{u}{x} \, dx.
\end{align*}
This leads to the reduction formula
\begin{align*}
u \; &= \; \frac{y^2}{ 6 f^2} \left[ w(z) + \frac{2}{3} f'^2 -
\left(\frac{2x}{y^2} f' + f''\right) f \right]\\
z \; &= \; \frac{f x}{y^2} - \frac{f'}{6},
\end{align*}
and the reduced equation is
\begin{equation} \label{eq:red03} 12 w - 12 z w'
+ w'^2 + (w + 24 z^2) w''=0.
\end{equation}
The solution for $w=0$ is
\begin{equation} \label{eq:sol03b} u =
\frac{y^2}{6 f^2}\left[ \frac{2}{3} f'^2 - \left( \frac{2 x}{y^2} f'
+ f'' \right) f \right].
\end{equation}

\begin{flushleft} {\bf Reduction by} $\set{\mathbf{x}_g, \mathbf{z}_f}$ \end{flushleft}
This Lie algebra belongs to type $\mathcal{L}_{2,10}.$ The
equivalent system is
\begin{align*}
dy \; &= \; \frac{2 y f'}{3 f} \, dt \\
du \; &= \; \frac{1}{6 f g} \left[-4 u g f' + 2 x f' g' - 2 x g f''
+ y^2 g' f'' - y^2 g f^{(3)}\right]\, dt - \frac{g'}{g} \, dx.
\end{align*}
The integrability condition is
\begin{align*}\label{eq:comp13}
  \quad &\quad  - 3 f g'^2 - g^2 f'' + 3 g (f' g' + f g'')=0.
\end{align*}
Integrating this last equation once yields
$$ 3 f \frac{g'}{g} - f' = K \equiv \text{constant}.  $$
We must therefore have
$$  g^3 = f \exp \left(\int \frac{K}{f} \, dt \right)  $$
as we already indicated earlier in ~\eqref{eq:comp13}

The reduction formula is
\begin{align*}
u\; &= \; \frac{1}{f^{2/3}} \left[  w(z) - \frac{x (K + f')}{3
f^{1/3}}
 - \frac{y^2}{6 f^{4/3}} \left( - \frac{K}{3} f' - \frac{2}{3} f'^2 + f f''
 \right)\right]\\
z \; &= \; y^3/ f^2,
\end{align*}

and the reduced equation is the linear equation
\begin{equation} \label{eq:red13} 9 z^{4/3} w'' +
6 z^{1/3} w' + \frac{K^2}{9}=0,
\end{equation}

with solution
\begin{equation*}
w(z)= -\frac{1}{18} K^2 z^{2/3} + 3z^{1/3} a_1 + a_2
\end{equation*}
 for some constants of integration $a_1$ and $a_2.$ Substituting this
 expression for $w(z)$ in the reduction formula yields
\begin{equation}
\label{eq:sol13}
\begin{split}
u &= \frac{1}{18 f^2}\left[  \left(18 a_2 + 54 a_1 (y^3/ f^2)^{2/3}
-
K^2 (y^3/f^2)^{2/3} \right) f^{4/3} \right]\\
  &+\frac{1}{18 f^2}\left( y^2 f' (K+ 2 f') - 3 f (2 K x
+ 2 x f'+ y^2 f'')\right).
\end{split}
\end{equation}

\begin{flushleft} {\bf Reduction by} $\set{\mathbf{y}_h, \mathbf{z}_f}$ \end{flushleft}
This Lie algebra is of type $\mathcal{L}_{2,12}.$ The equivalent
system is
\begin{align*}
dy \; &= \; \left[  \frac{2 y f'}{3 f} - \frac{h (2 x f' + y^2
f'')}{3 y f
h'}\right] dt + \frac{2 h}{y h'} \,dx \\
du \; &= \; \left[\frac{h''(2x f'+ y^2 f'')}{6 f h'} + \frac{(- 4 u
f' - 2 x f'' - y^2 f^{(3)})}{6 f} \right] \, dt - \frac{h''}{h'} \,
dx
\end{align*}
and the integrability condition is given by
\begin{equation}
\label{eq:comp23e} 3 f h'^2 + 2 h^2 f'' - 3h (f'h' + f h'')=0.
\end{equation}
Integrating this equation once gives
$$ \frac{2 f'-K}{3 f}= \frac{h'}{h},$$
where $K$ is a constant of integration. Thus we must have
\begin{equation}
\label{eq:comp23} h = C \left[  f^2 \exp \left(- \int \frac{K}{f} dt
\right) \right]^{1/3}.
\end{equation}
The reduction formula is
\begin{equation}\label{eq:fmlHF}
\begin{split}
u\; &= \; \frac{1}{f^{2/3}} \left[ w(z) - \frac{y^2 f'(K- 2 f')}{18
f^{4/3}} - \frac{- 6 K x + 6 x f' + 3 y^2 f''}{18 f^{1/3}} \right]\\
z \; &= \; \frac{12 x f + y^2 (K- 2 f')}{f^{4/3}}.
\end{split}
\end{equation}
These yield the reduced equation
\begin{equation} \label{eq:red23} K^2 + 90 K w' +
1296 w'^2 + 36 (36 w + K z)w''=0.
\end{equation}

For $K=0,$ this equation reduces to
\begin{equation} \label{eq:red230}
w'^2 + w w'' = 0
\end{equation}
with solution
\begin{equation} \label{eq:solr230} w= \beta \sqrt{2 z - \alpha}, \qquad
(\alpha \equiv
 \text{constant} , \quad \beta  \equiv \text{constant}).
\end{equation}

The corresponding solution of ~\eqref{eq:zk} is

\begin{equation} \label{eq:sol230} u= \frac{1}{18 f^2} \left[ 2 y^2 f'^2 - 3 f(2 x
f' + y^2 f'')+ 18 \beta f^{4/3} (4\frac{6 x f - y^2 f'}{f^{4/3}} -
\alpha)^{1/2} \right].
\end{equation}
All the exact solutions to Equation ~\eqref{eq:zk} that we have
found depend typically on zero arbitrary functions of time, and on
at most
 two such functions and some arbitrary constants.
Whenever the solution to a reduced equation was known, we were
always able to readily obtain the solution to the original equation
by a mere substitution into a reduction formula of the form
~\eqref{eq:u}. This situation is very different in the case of the
direct method of Clarkson and Kruskal that we investigate in more
details in the next section. Clearly, not all the group-invariant
solutions are physically relevant for the sound wave propagation
problem. However, they are certainly relevant for other problems
modeled by the same differential equations when the boundary
conditions change.
\section{Comparison with Clarkson and Kruskal's
method}\label{s:comparison}
 By the direct reduction method of
Clarkson and Kruskal \cite{clarkson}, all similarity solutions of
the  ZK equation of the form
\begin{equation} \label{eq:genformu}
u(t,x,y) = U(t,x,y, w(z)), \qquad z=z(t,x,y)
\end{equation}
where $U$ is a function of the indicated variables, and $w(z)$
satisfies an {\sc ode}, may be found by substituting Equation
~\eqref{eq:genformu} into Equation ~\eqref{eq:zk}. It is also argued
in \cite{zhu} that by a result of  Clarkson and Kruskal, in order to
find all such solutions, it is sufficient to look for $u$ in the
form
\begin{equation} \label{eq:linformu}
u(t,x,y)= \alpha + \beta w (z)
\end{equation}
where $\alpha$ and $\beta$ are functions of $t,x,$ and $y.$ The
latter equality is precisely our Equation ~\eqref{eq:u} that gives
the general form of all possible similarity solutions that we have
found thus far, in terms of the similarity variables. By making use
of the reduced anzatz $z_x \neq 0,$ Zhang {\em et al.} \cite{zhu}
showed in this way that the most general {\sc ode} satisfied by
$w(z)$ has the form

\begin{equation} \label{eq:zhu2}
w'^2 + ww'' + (A z + B) w' + 2 A w= \frac{1}{3}(A z+ B)^2
\end{equation}
and raised the question of whether there is any connection between
the direct method and the Lie classical method applied in the
preceding section.\par

One such connection could be determined by a  way of mapping the
reduced equation obtained by one method, to that obtained by the
other method, taking into account that the direct method gives in
principle the most general equation. Such a correspondence is not
obvious because of the various changes of variables through which
Equation ~\eqref{eq:zhu} was obtained, except perhaps if equations
are replaced by their equivalence classes under equivalence
transformations. However, due to a well-known result asserting that
two equivalent differential equations have isomorphic symmetry
groups \cite{olv2}, these equivalences become easier to establish.
The symmetry algebra of each of the reduced {\sc ode}s found in
section ~\ref{sb:2sbgp} has dimension at least two, except for
Equation ~\eqref{eq:red03} whose symmetry algebra is generated by
$\mathbf{v}_1= z
\partial_z + 2 w \partial_w$ alone. For $A \neq 0,$ Equation
~\eqref{eq:zhu2} is mapped after the change of independent variable
$y= z + B/A$ to the equation

\begin{equation} \label{eq:zhu2b}
w'^2 + w w'' + (A y) w' + 2 A w -\frac{1}{3} (A y)^2=0
\end{equation}

which has the same symmetry algebra as Equation ~\eqref{eq:red03}.
The necessary condition for an equivalence between these two
equations is thus satisfied. However, we've found that no linear
fractional transformation of the form
$$
H= \frac{a_1 w+ a_2}{a_3 w + a_4 }, \qquad  \xi= a_5 x+ a_6,
$$
where $a_1, \dots, a_6$ are arbitrary constants with $a_1 a_4- a_2
a_3 \neq 0,$ maps Equation ~\eqref{eq:zhu2b} to Equation
~\eqref{eq:red03}. This does not necessarily precludes the two
equations from being equivalent under other types of
transformations.\par

   For $A=0,$ Equation ~\eqref{eq:zhu2} reduces to
\begin{equation} \label{eq:zhuA0}
-\frac{B^2}{3} + B w' + w'^2 + w w''=0
\end{equation}
and the corresponding symmetry algebra has generators $\mathbf{v}_1=
\partial_z, \; \mathbf{v}_2= z \partial_z + w \partial_w.$ This symmetry
algebra has the same dimension as that for the reduced Equation
~\eqref{eq:red23} from section ~\ref{sb:2sbgp}. Equation
~\eqref{eq:red23} with $K=0$ is exactly the same as Equation
~\eqref{eq:zhuA0} with $B=0.$ For $K \neq 0,$ the change of
variables
$$ v= K m (36  w+ k z), \quad y= - K z/36, \quad \text{ with } m=
\frac{-35 + 3 \sqrt{21}}{36288}
$$
reduces Equation ~\eqref{eq:red23} to an equation of the form
$$ -\frac{B^2}{3} + B v' + v'^2 + v v''=0 $$
with $B= (-9 + 5 \sqrt{21}) K /672.$ This shows the equivalence of
the equations ~\eqref{eq:zhuA0} and ~\eqref{eq:red23}. \par

We thus see that by using the symmetry properties of the reduced
equations, we can always map each of the symmetry-reduced equations
to a sub-equation of ~\eqref{eq:zhu2}. It should be noted that
equations obtained by the direct method
   of Clarkson and Kruskal tend to be broad in nature
and hence more difficult to solve. In addition, no
   symmetry or other properties of the solutions these equations
   are known, and there is also no guarantee that the solutions
   to the original equation is recoverable from those of the reduced
   equation, because the process of reverting back the solution
   generally involves solving other {\sc pde}s. When this process does
give rise to an explicit solution of the original equation, it
usually involves relatively complicated expressions. For example,
the full expression of the exact solution generated by the reduced
Equation ~\eqref{eq:red230} and given by ~\eqref{eq:solr230} is
easily found in Lie's approach and depends only on one arbitrary
function of time and two arbitrary constants. The same solution is
obtained by the direct method in \cite{zhu} only through the solving
a number of intermediary {\sc pde}s, and depends on no less than
four arbitrary functions of time plus the two arbitrary
constants.\par

   As for the reduction to $(1+1)$-dimensional equations, the similarity
solutions and corresponding reduced equations found in Section
~\ref{sb:1sbgp} are much simpler and complete than those obtained
using the direct method of Clarkson and Kruskal in \cite{zhu}, where
both the equations and corresponding solutions depend on up to four
arbitrary functions and are also  usually defined only implicitly in
terms of the solutions of some {\sc pde}s. For instance, if a
general reduced equation of the form ~\eqref{eq:redf} were to be
determined by the direct method, it would more likely be determined
only implicitly and be expressed in terms of unnecessary arbitrary
functions. Although the direct method also gives in principle the
similarity solution as well as the reduced equation, its algorithm
is relatively complicated  and yields most often complicated and
incomplete results. \par

In the classical Lie reduction approach, the symmetry properties of
each similarity solution is known and in some cases, especially in
the case of solvable symmetry algebras, the solution to the original
equation can be recovered by quadratures. \par

It should however also be noted that even with Lie's classical
method, having to revert back to the solution of the original
equation by quadratures alone does not guarantee the explicit
determination of this solution. Indeed, these quadratures may
involve transcendental or integral equations that aren't easy to
solve. This fact is common in the determination of group-invariant
solutions and we give another example here by attempting to solve
Equation ~\eqref{eq:red23}, for which the symmetry generators are
given by
   $$
   \mathbf{v}_1= (-36/k) \partial_z + \partial_w, \quad \text{ and
   }\quad
   \mathbf{v}_2= z \partial_z + w \partial_w.
   $$
These vector fields clearly generate a solvable Lie algebra with
commutation relations $[ \mathbf{v}_1,  \mathbf{v}_2]=
\mathbf{v}_1.$ In terms of the rectifying coordinates
\begin{equation} \label{eq:ry}
y= 36 w+ K z  \qquad \text{ and  }  r= -(K x)/ 36
\end{equation}
 for $ \mathbf{v}_1,$ Equation
~\eqref{eq:red23} reduces to

\begin{align}
- H + 18 H^2 + 648 H^3 + y  H'&=0  \label{eq:red23r1}\\[-7mm]
\intertext{ \vspace{-3mm} where } H &= \frac{d r}{d y}.
\label{eq:drdy}
\end{align}

Equation ~\eqref{eq:red23r1} retains the symmetry $\mathbf{v}_2,$
which is given by $\tilde{\mathbf{v}}_2= y \partial_y + r
\partial_r$ in terms of the variables $r$ and $y,$ and by
\begin{equation} \label{eq:red23v2}
\hat{\mathbf{v}}_2= \frac{H}{(1+ 18 H)^{1/3}(36 H-1)^ {2/3}}\,
\partial_y
\end{equation}
in terms of $H$ and $y.$ In terms of the rectifying coordinates
\begin{align}
 \xi &= H  \\
S &= \frac{(1+ 18 H)^{1/3}(36 H-1)^ {2/3}}{H} y \label{eq:S}
\end{align}
for $\hat{\mathbf{v}}_2,$  Equation ~\eqref{eq:red23r1} reduces to
$$  \xi^3 (1+ 18 \xi)^{5/3}(36 \xi -1)^{4/3} S'=0$$
with solution $S(\xi)= a,$ where $a$ is a constant of integration.
It thus follows from Equation ~\eqref{eq:S} that the solution $H=
F(y)$ of ~\eqref{eq:red23r1} is given implicitly by the polynomial
equation
\begin{equation} \label{eq:poly3}
a H^3 - y^3 + 54 H y^3 - 23328 H^3 y^3=0.
\end{equation}
Assuming that the function $F(y)$ is known, it readily follows from
Equation ~\eqref{eq:drdy} that

\begin{equation} \label{eq:rint}
r= \int F(y)dy \; = \; G(y)
\end{equation}
for a certain function $G(y),$  and the substitution of
~\eqref{eq:ry} and ~\eqref{eq:fmlHF} into this last equation leads
to the solution to the original equation by an additional
quadrature. The problem here with these quadratures is that the
integral $r= \int F(y)dy$ is not easy to solve because of the
complicated form of the function $F.$ Indeed, one of the simplest
roots $H=F(y)$ of the polynomial Equation ~\eqref{eq:poly3} is given
by
$$  F(y)= \frac{-36 a y^3 + 839808 y^6 + 2^{1/3}(y^3 X^2 + \sqrt{a y^6 X^3})^{2/3}   }
{2^{2/3} X (y^3 X^2 + \sqrt{a y^6 X^3})^{1/3}}$$
 where $X= a- 23328 y^3,$ and so the integral equation $r= \int F(y)
 dy$ does not appear to be obvious to solve, unless perhaps an
 appropriate change of variable to be found was performed. In
 reality, the difficulties with these quadratures are normally to be
 associated with the equation itself and not with Lie's reduction
 method.

\section{Concluding remarks}
In this paper we have given a classification of low-dimensional
subalgebras of the ZK symmetry algebra into conjugacy classes and
determined  similarity reductions of the ZK equation to
$(1+1)$-dimensional equations and to {\sc ode}s. We have thus
derived a number of new exact solutions to this equation in this
way. Consequently, we have been able to compare the results obtained
using Lie's classical method with those obtained in \cite{zhu} with
the direct method of Clarkson and Kruskal, and found in particular
equivalence transformations between the reduced equations obtained
by this two methods. This investigation shows that not only Lie's
algorithm is simpler and much richer in properties, but it also
yields simpler and more complete results as opposed to the direct
method where both solutions and reduced equations are most often
determined only implicitly by complicated expressions.\par
The example of  reduction by direct case analysis treated in Section
~\ref{ss:byanal}   suggests that equations obtained by the direct
method can also be viewed as reduced equations obtained by the Lie
classical method by performing a reduction with a system of generic
generators corresponding to Lie algebras having orbits of fixed
dimensions, rather than performing the reduction with some specific
generators.

\end{document}